\documentclass[runningheads]{llncs}

\usepackage[T1]{fontenc}

\usepackage{graphicx}
\usepackage{amsmath}
\usepackage{amssymb}
\usepackage{amsfonts}
\usepackage{etoolbox}
\usepackage{float}

\AtEndEnvironment{proof}{\phantom{}\qed}

\newcommand{\Z}{\mathbb{Z}}
\newcommand{\N}{\mathbb{N}}
\newcommand{\abs}[1]{\left\vert{#1}\right\vert} 
\newcommand{\hole}{\diamondsuit}

\begin{document}

\title{A Connection Between Unbordered Partial Words and Sparse Rulers\thanks{Supported by the Academy of Finland under grants 339311 and 346566}}

\author{Aleksi Saarela  \and
Aleksi Vanhatalo}

\institute{Department of Mathematics and Statistics, University of Turku, Finland, \texttt{aleksi.saarela@utu.fi}, \texttt{aleksi.e.vanhatalo@utu.fi}}

\maketitle             

\begin{abstract}
    $\textit{Partial words}$ are words that contain, in addition to letters, special symbols $\diamondsuit$ called $\textit{holes}$. Two partial words of $a=a_0 \dots a_n$ and $b=b_0 \dots b_n$ are $\textit{compatible}$ if for all $i$, $a_i = b_i$ or at least one of $a_i, b_i$ is a hole. A partial word is $\textit{unbordered}$ if it does not have a nonempty proper prefix and a suffix that are compatible. Otherwise the partial word is $\textit{bordered}$.
    A set $R \subseteq \{0, \dots, n\}$ is called a $\textit{complete sparse ruler of length $n$}$ if for all $k \in \{0, \dots, n\}$ there exists $r, s \in R$ such that $k = r - s$. These are also known as $\textit{restricted difference bases}$.
    From the definitions it follows that the more holes a partial word has, the more likely it is to be bordered. By introducing a connection between unbordered partial words and sparse rulers, we improve bounds on the maximum number of holes an unbordered partial word can have over alphabets of sizes $4$ or greater. We also provide a counterexample for a previously reported theorem.
    We then study a two-dimensional generalization of these results. We adapt methods from one-dimensional case to solve the correct asymptotic for the number of holes an unbordered two-dimensional binary partial word can have. This generalization might invoke further research questions.
\end{abstract}

\section{Introduction}
Partial words are words with special hole symbol, denoted usually by $\hole$. They were introduced, together with a new word relation where holes are considered to be compatible with any letter, by Berstel and Boasson in \cite{berstel1999partial}. They wanted a systematic, well-defined way to study words with some letters missing and replaced by ``do not know'' symbols. Since then many classical word combinatorics results have been considered in context of partial words and this new relation, see for example \cite{halava2009overlap} and \cite{halava2008square}.

In \cite{blanchet2009many} and \cite{Combinatorics-on-partial-word-borders} Blanchet-Sadri, Allen et al. considered unbordered partial words with maximum number of holes. They were able to find the closed form for this maximum in binary alphabet, a general upper bound depending of the size of the alphabet and lower bounds in alphabet sizes $3$ and $4$. 

Sparse rulers are seemingly unrelated combinatorial objects. The intuitive idea of sparse ruler is to have a ruler with low number of marks which is still able to measure every distance less than its length between its marks. Sparse rulers have been studied by number theorist atleast from late 1940s, for example in \cite{erdos1948representation}, \cite{redei1949representation} and \cite{Leech} under the name of restricted difference bases.

In section $3$ we show a connection of unbordered partial words and sparse rulers. Using results from the theory of rulers we are then able to correct and improve some results on unbordered partial words. In particular, we get a new best upper and lower bounds for the number of holes in large enough alphabets. As a necessary tool we also provide a fairly simple proof for the correctness of the currently best known ruler construction, something which was originally omitted by Wichmann in \cite{Wichmann}. The connection and some its corollaries were first reported in second author's master's thesis \cite{Gradu} (in Finnish).

In section $4$ we consider the two-dimensional case of unbordered partial words and rulers. We solve the correct asymptotic for the number of holes in $n$ by $m$ unbordered binary partial words, and prove some other preliminary results. While 2D sparse rulers have known applications in sensor placements, see \cite{sparsearraysovellus} and \cite{soveltajat2dviivoittimet}, we are not aware of any pure mathematical research on the subject.

\section{Partial words}
In this section we present the basic definitions concerning partial words. We use slightly simplified version of some definitions, but all equivalent to more commonly used variants.
\begin{definition}
    Partial word over an alphabet $A$ is a word over the alphabet $A_\hole = A \cup \{\hole\}$ where $\hole \notin A$. The symbol $\hole$ is called a hole and is thought to be a ``do not know'' symbol in a word over $A$. The set of all finite partial words is denoted by $A_\hole^*$.
\end{definition}

Partial words without holes (that is, words over the alphabet $A$) are full words. The length of a partial word is measured similarly to a full word. We call all elements in $A \cup \{\hole\}$ symbols while only elements in $A$ are called letters.

\begin{definition} Let $w=w_0w_1 \ldots w_n$ be a partial word over $A$.
Denote by $\mathrm{D}(w)$ the set $\{i_0, i_1, \ldots i_m\}$ where $w_{j} \in A \iff j=i_k$ for some $k$. The set $\mathrm{D}(w)$ is called the domain of $w$.
\end{definition}
Now we can define the compatibility relation between partial words.
\begin{definition}
    Partial words $w=w_0w_1\ldots w_n$ and $v=v_0v_1\ldots v_m$ are compatible if $n=m$ and for all $i \in \mathrm{D}(w) \cap \mathrm{D}(v)$ we have $w_i = v_i$. 
    This is denoted by $w \uparrow v$.
\end{definition}
    In intuitive sense, this means that compatible partial words could be equal if all the unknown symbols could be identified.
\begin{example}
    Let $A=\{a,b,c\}$. Then partial words $abc$, $\hole cc$ and $a\hole c$ are all elements of $A_\hole^*$, have all length $3$ and $abc \uparrow a\hole c$, $\hole cc \uparrow a\hole c$ but $\hole cc \not\uparrow abc$.
\end{example}

The main topic of this paper can now be defined.

\begin{definition} Border of a partial word $w$ is a pair $(x,y)$ such that $x$ is a proper prefix of $w$, $y$ is a proper suffix of $w$ and $x \uparrow y$. A partial word without any borders is called unbordered partial word.
\end{definition}

From the definitions it is clear that the more holes a partial word has, the more likely it is to be bordered. This leads to interesting question of maximum number of holes in length-$n$ unbordered partial word over a given alphabet. Adding to the notation in \cite{Combinatorics-on-partial-word-borders} we also want a notation in the case of arbitrary alphabet.

\begin{definition}
 $\mathrm{HB}_k(n)$ denotes the maximum number of holes an unbordered partial word of length $n$ can have over an alphabet of size $k$. Similarly, $\mathrm{HB}_\infty (n)$ denotes the maximal number of holes an unbordered partial word can have over an infinite alphabet, i.e. $\mathrm{HB}_\infty (n) = \mathrm{HB}_n(n)$.
\end{definition}

In \cite{Combinatorics-on-partial-word-borders} and \cite{blanchet2009many}, Allen, Blanchet-Sadri and others provived a closed form of $\mathrm{HB}_2$ and bounds for larger alphabet sizes. The bounds we state in Theorem $\ref{original bounds}$ are a simplified version of the ones in \cite{blanchet2009many}. In some particular cases of $n$ the lower bound for $\mathrm{HB}_4$ is better in \cite{blanchet2009many}, but still something we will improve upon in chapter $3$.

\begin{theorem}[\cite{Combinatorics-on-partial-word-borders} and \cite{blanchet2009many}]
\label{original bounds}
For all $n> 9$ and $k\geq 4$,
    $$ n- \lceil 2\sqrt{n+3} \rceil +2 \leq \mathrm{HB}_3(n)\leq \mathrm{HB}_4(n) \leq \mathrm{HB}_k(n) \leq n- \sqrt{\frac{2k}{k-1}(n-1)}.$$
\end{theorem}

\section{The connection with sparse rulers}
We show a new connection between unbordered partial words and an independently studied combinatorical problem of sparse rulers. This proves to be very fruitful connection as the theory of rulers has more powerful theorems than the partial words counterparts.

\begin{definition} 
A subset $R \subseteq \{0,1,\ldots n\}$ containing the numbers $0$ and $n$ is a sparse ruler (or just a ruler) of length $n$. The elements of a ruler are called marks. The ruler $R$ can measure the distance $d$ if there exists $i \in R$ and $j \in R$ so that $\abs{i-j}=d$. If the ruler can measure all distances $d \in \{0,1,\ldots n\}$, then the ruler is complete.
\end{definition}
\begin{example}
    The set $\{0,1,4,6\}$ can easily be verified to be an example of a complete sparse ruler of length $6$.
\end{example}
The combinatorial optimization problem that comes with complete sparse rulers is to minimize the number of marks needed to construct a complete sparse ruler with given length $n$.
\begin{definition}
    $\mathrm{M}_1(n)$ is the minimum number of marks a length $n$ complete sparse ruler can have.
\end{definition}
 The following lower bound for the function $\mathrm{M}_1(n)$ was originally found by R\'edei and R\'enyi \cite{redei1949representation} in 1949 and then improved marginally by Leech \cite{Leech} in 1956 and then again by Bernshteyn and Tait \cite{bernshteyn2019improved} in 2019. We present the lower bound by Leech. For the upper bounds, see chapter about Wichmann-rulers.

\begin{theorem}[\cite{Leech} and \cite{redei1949representation}] \label{ruler lower bound}
$$\mathrm{M}_1(n) \geq \max_{0<\delta<2\pi} \sqrt{2n+1- \frac{\sin{\delta}}{\sin{\frac{\delta}{2n+1}}}}>\max_{0<\delta<2\pi}\sqrt{2(1-\frac{\sin{\delta}}{\delta})n +1-\frac{\sin{\delta}}{\delta}}.$$ 
Or approximately 
$$\mathrm{M}_1(n)>\sqrt{2.43n}.$$

\end{theorem}

We now show how the rulers are linked to the unbordered partial words. A part of the proof is equivalent to an earlier result by Allen et al. (\cite{weakperiod}, Proposition 1.), but there was no mention of connection to sparse rulers.

\begin{theorem} For all unbordered partial words $w$, the domain $\mathrm{D}(w)$ is a complete sparse ruler of length $\abs{w}-1$. Conversely, for all complete sparse rulers $R$, there exist an unbordered partial word $w$ over some alphabet such that $R=\mathrm{D}(w)$.
\label{correspondense theorem}
\end{theorem}
\begin{proof}
    Let $w=w_0w_1\ldots w_{n-1}$ be an unbordered partial word and let $0<l < n$.
    Since $w$ has no borders, its prefix $x=w_0 \ldots w_{l-1}$ and suffix $y=w_{n-l} \ldots w_{n-1}$ of length $l$ are uncompatible. This means that for some $i\leq l-1$ we have $w_i \not \uparrow w_{n-l+i}$ and thus $w_i \neq \hole$ and $w_{n-l+i} \neq \hole$. This means that $i$ and $n-l+i$ are elements of the set $\mathrm{D}(w)$. Since $n-l+i-i=n-l$, the set $\mathrm{D}(w)$ can represent the distance $n-l$ using differences. Since $l \in \{1,2,\ldots,n-1\}$ was arbitrary, $\mathrm{D}(w)$ can represent all numbers $n-l \in \{1,2,\ldots n-1\}$ by differences and thus $\mathrm{D}(w)$ is a complete sparse ruler of length $n-1 = \abs{w}-1$.

    For the converse statement, assume $R=\{r_0, r_1, r_2, \ldots r_{m-1}\}$ to be a complete sparse ruler of length $n-1$ with $m$ marks. We construct a partial word $v = v_0 v_1 \ldots v_{n-1} $ over the alphabet $R$ such that 
    $$ v_i = 
    \begin{cases}
        i, & \text{if } i \in R \\
        \hole, & \text{otherwise}
    \end{cases}$$
    Now clearly $\mathrm{D}(v)=R$. Now let $l$ be a length of a proper prefix $x$ of $v$ and let $y$ be the same length suffix. Since ruler $R$ can represent the length $n-l$ by difference, there are marks $r_i$ and $r_j=r_i + n-l$ in the ruler $R$. Thus in the word $v$ there are different letters placed at the index positions $r_i$ and $r_i + n - l.$ But the index $r_i + n -l$ in a word $v$ is the same position as the index $r_i$ of the suffix $y$, and thus the prefix $x$ and suffix $y$ of the word $v$ have both a letter in the index position $r_i$. This means that $x \not \uparrow y$ and the word $v$ has no border of length $l$. Since $l$ was arbitrary, $v$ is unbordered.
\end{proof}

\begin{figure}
    \centering
    \includegraphics[width=\textwidth]{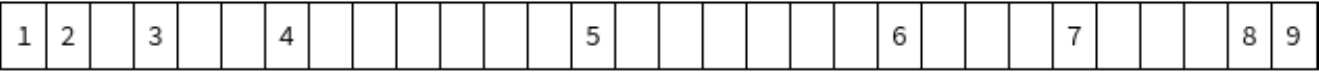}
    \caption{Ruler $\{0, 1, 3, 6, 13, 20, 24, 28, 29\}$ of length $29$ and equivalent unbordered partial word of length $30$ over the alphabet $\{1,2,3,\ldots,9\}$. Here the hole symbol is just an empty space.}
    \label{rulerwordfigure}
\end{figure}

Figure $\ref{rulerwordfigure}$ demonstrates the two-way correspondence. This connection is particularly interesting when considering the maximum number of holes an unbordered partial word can have, since (without fixed alphabet) the question can be restated as how many marks does a length $n$ complete sparse ruler need. 

\begin{corollary} For all $k \geq \mathrm{M}_1(n-1)$, 
$$\mathrm{HB}_k(n) = \mathrm{HB}_\infty (n) = n- \mathrm{M}_1(n-1).$$
\label{optimal corresponding}
\end{corollary}
\begin{proof}
    Let $R$ be a length $n-1$ complete sparse ruler with $\mathrm{M}_1(n-1)$ marks. By Theorem $\ref{correspondense theorem}$ we have corresponding unbordered partial word $w$ of length $n$ such that $\mathrm{D}(w)=R$ and so that $w$ is over an alphabet of size $\mathrm{M}_1(n-1)$. The number of holes in $w$ is $n-\mathrm{M}_1(n-1)$. Now assume that there exist an unbordered partial word $w'$ of length $n$ with at least $n-\mathrm{M}_1(n-1) +1$ holes. Then by Theorem $\ref{correspondense theorem}$, $\mathrm{D}(w')$ is a complete sparse ruler of length $n-1$. The number of elements in $\mathrm{D}(w')$ is at most $\mathrm{M}_1(n-1)-1$, which is a contradiction with the definition of $\mathrm{M}_1(n-1)$. This completes the proof.
\end{proof}

Note that the equality might already hold for smaller alphabet sizes. In fact, we didn't find any case where alphabet size of more than $4$ is needed. This observation is partly explained later in the chapter about Wichmann words.

\subsection{Applying ruler theory to unbordered partial words}

In this section, by Corollary \ref{optimal corresponding} we can directly translate some results of the function $\mathrm{M}_1$ into results of the functions $\mathrm{HB}_k$. It so happens that the study of rulers gives improved bounds on $\mathrm{HB}_k$ when $k \geq 6$ for upper bound and $k\geq 4$ for the lower bound. The lower bound is obtained in the next section about Wichmann rulers.

The new upper bound is obtained easily.
\begin{theorem}
    For all $k\geq 2$ and $n\geq 1$ 
    $$ \mathrm{HB}_k(n) \leq n- \sqrt{2.43(n-1)}.$$
    \end{theorem}
    \begin{proof}
        Clear result from Theorem $\ref{ruler lower bound}$ and Corollary $\ref{optimal corresponding}$.
    \end{proof}

Before going for the lower bound, which needs some additional work, we show the following unintuitive result from the known values of the function $\mathrm{M}_1$. This is direct counter-example to a result that was thought to be already proven in \cite{Combinatorics-on-partial-word-borders}.

\begin{corollary} For some $n \in \N$, $\mathrm{HB}_k(n+1)>\mathrm{HB}_k(n)+1$ for all $k \geq 4$.
\end{corollary}
\begin{proof}
    From known values of the function $\mathrm{M}_1(n)$ (OEIS \cite{ruler-sequence}) we see that $\mathrm{M}(135)=21$ and $\mathrm{M}_1(138)=20$. Then by Corollary $\ref{optimal corresponding}$ we have 
    \begin{align*}
        \mathrm{HB}_{21}(136) &= 136 -21 = 115 \text{ and }\\
        \mathrm{HB}_{21}(139) &= 139 -20 = 119.
    \end{align*}
    The following unbordered words prove that these values are already obtained in four letter alphabet:
    $$a^4b^3\hole^{58}a(\hole^2 c)^3 (\hole^6 d)^7 (\hole^3 b)^3 $$ and $$ab^3\hole^3 (a\hole^6)^3 b (\hole^{14} c)^5 (\hole^7 d)^4 b^3.$$
    
    So $\mathrm{HB}_{4}(139)>\mathrm{HB}_{4}(136)+3$ which is equivalent to the claim. 
    
\end{proof}

There is currently no intuitive nor rigorous explanation why $\mathrm{M}_1$ is non-increasing and thus these particular values of $\mathrm{HB}_4$ are bit of a mystery. The result $\mathrm{HB}_k(n+1) \leq \mathrm{HB}_k(n)+1$ holds for $k=2$ due to the closed form of the function $\mathrm{HB}_2$ \cite{blanchet2009many}. This is an open problem for $k=3$.

\subsection{Wichmann-rulers and Wichmann-words}
The best upper bound for the function $\mathrm{M}_1$ is by Ed Pegg Jr. (\cite{hittingallthemarks} 2020) with construction for large values of $n$ and by computer search for ``small'' values of $n$ ($n<257 992$). The construction for large values depends on so called Wichmann rulers, constructed by Wichmann in 1963 \cite{Wichmann}. In this chapter, as the references are lacking, we give a self-contained presentation of Wichmann ruler construction. We then show how to turn this construction to unbordered partial words over the minimal alphabet size.

If $R = \{a_0, \dots, a_k\}$ is a sparse ruler, we say that
\begin{math}
    D = (a_1 - a_0, \dots, a_k - a_{k - 1})
\end{math}
is a \emph{difference representation} of $R$.
If $a_0 < \dots < a_k$, then the elements of $D$ are all positive.
However, this is not required,
which means that a ruler has many different difference representations.
On the other hand, each difference representation specifies a unique ruler since $0$ is the smallest element in rulers.

\begin{definition}
    Wichmann ruler $w_1(r,s)$ is the ruler defined by difference representation
    $$(1^r, r+1, (2r+1)^{r}, (4r+3)^{s}, (2r+2)^{r+1},1^r) $$ or equivalently 
\begin{equation} \label{eq:wichmann2}
    ((-1)^r, (2r + 1)^{r + 1}, (4r + 3)^s, (2r + 2)^{r + 1}, 1^r),
\end{equation} where $a^b$ means that the value $a$ is repeated $b$ times.
\end{definition}

For example, the ruler $w_1(1,1)$ has the difference representation $$(1,2,3,7,4,4,1)$$ and thus $w_1(1,1)=\{0,1,3,6,13,17,21,22 \}$. The difference representation $\eqref{eq:wichmann2}$ might better show the symmetry of the construction. For $w_1(1,1)$, it would be $(-1,3,3,7,4,4,1).$

For most lengths, there is no Wichmann ruler, so we will also need extended Wichmann rulers.
\begin{definition}
    Extended Wichmann ruler $w_1(r,s,i,j)$ is the ruler defined by the difference representation 
    $$(1^r, r+1, (2r+1)^{r}, (4r+3)^{s}, (2r+2)^{r+1},1^r,(r+1)^i,j) $$ or equivalently
    $$ ((-1)^r, (2r + 1)^{r + 1}, (4r + 3)^s, (2r + 2)^{r + 1}, 1^r, (r+1)^i,j) $$ where $j\leq r+1$ and $a^b$ means that the value $a$ is repeated $b$ times.
\end{definition}
Extended Wichmann rulers are complete. Wichmann originally omitted the proof due to its tedious nature, so we include it to cut loose ends in our theory later. Reader can easily skip this proof and accept that the extended Wichmann rulers are complete.

The proof is greatly improved version of the one in second writer's master's thesis \cite{Gradu} (in Finnish).

\begin{theorem}
\label{complete rulers}
    For all $r,s,i \in \N$ and $j\leq r+1$ the rulers $w_1(r,s,i,j)$ are complete.
\end{theorem}
\begin{proof}
    Firstly we note that if $w_1(r,s)$ is complete, then is also $w_1(r,s,i,j)$. This is because the ruler $w_1(r,s,i,j)$ can measure the distances that are less or equal to the length of the ruler $w_1(r,s)$ by completeness of $w_1(r,s)$. Then for the longer lengths of the type $n+l$, where $n$ is the length of of the ruler $w_1(r,s)$ and $l\leq (r+1)i+j$, the ruler $w_1(r,s,i,j)$ has the marks $\{0,1,\ldots, r\}$ and marks $\{n+t(r+1) \ \vert \ t\leq i \}$ and the mark $n+i(r+1)+j$.

    What is left to prove is that $w_1(r,s)$ is complete for all parameter values. So we need to prove that $w_1(r, s)$ can measure every distance from $1$ to its length,
    which is $(4r + 3)(r + s + 1) + r$.
    It is sufficient to prove that for all $x \leq (4r + 3)(r + s + 1) + r - 1$,
    if $x$ can be measured, then so can $x + 1$.
    
    For rest of the proof, we use the fact that a sparse ruler with a difference representation $(b_1, \dots, b_k)$ can measure a distance $x \geq 1$ if and only if there exists $i \leq j$ such that 
\begin{math}
        x = |b_i + \dots +  b_j|.
\end{math}
    So in our case $x$ can be measured if and only if it is a sum of consecutive elements of \eqref{eq:wichmann2}.
    If the first of these elements is $-1$,
    then we can leave it out to get a representation for $x + 1$.
    Similarly, if these elements are followed by a $1$,
    then we can add it to the end to get a representation for $x + 1$.
    We still have to consider sums of consecutive elements of the part
    \begin{equation*}
        ((2r + 1)^{r + 1}, (4r + 3)^s, (2r + 2)^r)
    \end{equation*}
    and sums that reach all the way to the end of \eqref{eq:wichmann2} and end in $r$ copies of 1.
    For the former sums, we have the following cases:
    \begin{itemize}
        \item[$\bullet$]
        If $x = i (2r + 1) + j (4r + 3)$, where $i \geq 2$ and $j < s$,\\
        then $x + 1 = (i - 2) (2r + 1) + (j + 1) (4r + 3)$.
        \item[$\bullet$]
        If $x = 2r + 1 + j (4r + 3)$,\\
        then $x + 1 = j (4r + 3) + 2r + 2$.
        \item[$\bullet$]
        If $x = i (2r + 1) + s (4r + 3) + j (2r + 2)$, where $i \geq 1$ and $j \leq r$,\\
        then $x + 1 = (i - 1) (2r + 1) + s (4r + 3) + (j + 1) (2r + 2)$.
        \item[$\bullet$]
        If $x = i (4r + 3) + j (2r + 2)$, where $i \geq 1$ and $j \leq r - 1$,\\
        then $x + 1 = (i - 1) (4r + 3) + (j + 2) (2r + 2)$.
        \item[$\bullet$]
        If $x = i (4r + 3) + r (2r + 2)$,\\
        then $x + 1 = r (-1) + (r + 1) (2r + 1) + i (4r + 3)$.
    \end{itemize}
    For the latter sums, we have the following cases:
    \begin{itemize}
        \item[$\bullet$]
        If $x = i (2r + 1) + s (4r + 3) + (r + 1) (2r + 2) + r$, where $i \leq r$,\\
        then $x + 1 = i (-1) + (r + 1) (2r + 1) + s (4r + 3) + (i + 1) (2r + 2)$.
        \item[$\bullet$]
        If $x = i (4r + 3) + (r + 1) (2r + 2) + r$, where $i \leq s - 1$,\\
        then $x + 1 = r (2r + 1) + (i + 1) (4r + 3)$.
        \item[$\bullet$]
        If $x = i (2r + 2) + r$, where $i \leq r$,\\
        then $x + 1 = (r - i) (-1) + (i + 1) (2r + 1)$.
    \end{itemize}
    This completes the proof.
\end{proof}

We can now present the proof for asymptotic growth on the number of marks that Wichmann-rulers can achieve. Proof is the same as in Wichmann's original paper, but is included so a reader interested in unbordered partial words can follow the calculations later with words in mind. The constant term in the bound could be lowered according to Pegg's result \cite{hittingallthemarks}, but we mostly care about the asymptotic behaviour.

\begin{theorem}[\cite{Wichmann}] \label{useable upper bound} For all $n\geq 213$ there exists an extended Wichmann-ruler $w_1(r,s,i,j)$ of length $n$ so that 
    $$\abs{w_1(r,s,i,j)} \leq  \sqrt{3n} + 4.$$
    \end{theorem}
    
    \begin{proof}
    The lenght of a Wichmann-ruler $w_1(r,s)=w_1(r,s,0,0)$ is $l_{r,s}=4r(r+s+2) + 3s + 3$ and the number of marks is $m_{r,s}=4r + s + 3$. Let $r\geq 4$ and $s=2r+e \in [2r-2,2r+4]$. Then 
    \begin{align*}
         \sqrt{3l_{r,s}} =&\sqrt{12r(r+s+2) + 9s + 9}  \\
        =& \sqrt{36r^2 + 2(e+3)6r + 6r + 9e +9 }  \\
        \geq& \sqrt{36r^2 + 2(e+3)6r + (e+3)^2 }  \\
        =&6r + e + 3 = m_{r,s}.
    \end{align*}
    
    An increase of parameter $s$ by one will increase the length of Wichmann-ruler by $4r+3 < 4(r+1)$. So for all $n \in [l_{r,2r+e}, l_{r,2r+e+1}]$ we can extend the Wichmann-ruler $w_1(r,2r+e)$ with at most $4$ marks to obtain a complete ruler of length $n$. This extended Wichmann-ruler will have at most $\sqrt{3l_{r,2r+e}}+4 \leq \sqrt{3n}+4$ marks. To finish the proof a simple calculation shows 
    us that $l_{r,2r+4} = l_{r+1,2(r+1)-2}$, meaning that all $n \geq l_{4,6} = 213$ are between these good Wichmann-rulers and thus there exists that length extended Wichmann-ruler with the claimed amount of marks.
\end{proof}
We now have (mostly) proven
\begin{theorem}[\cite{Wichmann}]
For all $n$,
    $$\mathrm{M}_1(n) \leq  \sqrt{3n} + 4.$$
\end{theorem}
\begin{proof}
    For $n\geq 213$ this follows from Theorems $\ref{complete rulers}$ and $\ref{useable upper bound}$. For $n<213$, the constant term $+4$ is so generous that the claim can easily be checked by computer search, or from known values of $\mathrm{M}_1$ in OEIS \cite{ruler-sequence}.
\end{proof}

Now we can consider how to turn Wichmann-rulers into unbordered partial words. Theorem $\ref{wichmannword4}$ gives the minimum alphabet size where this can be done. We call the constructed words $\textit{Wichmann words}$.

\begin{theorem} Partial words $$W(r,s)=ab^{r}\hole^r (a\hole^{2r})^{r}b (\hole^{4r+2}c)^{s} (\hole^{2r+1}d)^{r+1}b^r$$ and $$W(r,s,i,j)=ab^{r}\hole^r (a\hole^{2r})^{r}b (\hole^{4r+2}c)^{s} (\hole^{2r+1}d)^{r+1}b^r(\hole^r c)^i \hole^{j-1} c$$ are unbordered for all parameter values $r,s,i\in \N$ and $1\leq j\leq r+1$.
\label{wichmannword4}
\end{theorem}
    \begin{proof}
        In this proof, instead of directly showing there is no border, we equivalently show that the domain of these partial words are complete sparse rulers in such a way that the measuring can be done between marks that correspond to different letters in the word. The word is then unbordered similarly to the word in the proof of Theorem $\ref{correspondense theorem}$.
        
        For all $i$ and $j \leq r+1$ a partial word $W(r,s,i,j) = W(r,s) (\hole^r c)^i \hole^{j-1} c$. Then it is quite easy to see (following the first part of Theorem $\ref{complete rulers}$) that if $W(r,s)$ is unbordered, then so is $W(r,s,i,j)$. So we concentrate on $W(r,s)$ for the rest of the proof.

        $W(r,s)$ is constructed so that $\mathrm{D}(W(r,s)) = w_1(r,s)$. We label the defining vector of differences in Wichmann-rulers as follows: $$(1^r, r+1, (2r+1)^{r}, (4r+3)^{s}, (2r+2)^{r+1},1^r) = (A,B,C,D,E,F), $$ where $A=1^r, B=r+1, C=(2r+1)^r$ and so on. Now note that if a distance is measured using two marks within a single area, the measured distance is a multiple of that areas difference. Thus each area needs to use only at most 2 letters to be able to measure within itself. Then all other measuring the underlying Wichmann ruler does is with marks in 2 different areas. If the areas use disjoint alphabets, then those areas clearly are able to measure all the differences between them. Also note that area $B$ is fully covered by areas $A$ and $C$, so it does not need a separate check. Above observations take care of all other distances but the distances from area $A$ to areas $C$ and $F$ and from area $C$ to the area $F$. The fact that the word measures these has to be checked separately.
    
        The distance between the starting $a$ in the area $A$ to the other $a$ letters in the area $C$ is a multiple of $2r+1$. These distances can be measured within the area $C$.

        The distances between letters $b$ in the area $A$ and the letter $b$ in the area $C$ is the range $[r+1+r(2r+1),2r +r(2r+1)]=[r(2r+2)+1,r(2r+2)+r]$. This range can be also measured from the area $E$ to the the area $F$.
        
        The measurements from the area $A$ to the area $F$ can always be made using the letter $a$ from the area $A$ or by using the $d$ from the area $F$ so the common letter $b$ does not matter.
        
        Lastly, the distances between the last letter of the area $C$ and the ending letters of the area $F$ are the range $[(4r+3)s + (2r+2)(r+1)+1,(4r+3)s + (2r+2)(r+1)+r]$ which is equal to $[r+1+(2r+1)r+(4r+3)s + (2r+2),r+r+1+(2r+1)r+(4r+3)s + (2r+2)]$. This can be measured form the area $A$ to the second mark in the area $E$.

        Now this word can measure all the distances similarly to a Wichmann-ruler, so it is unbordered.

    \end{proof}

Figure $\ref{wichmannword22}$ shows an example of the construction given in Theorem $\ref{wichmannword4}$. We now get the new lower bound for $\mathrm{HB}_4$.
\begin{figure}
    \centering
    \includegraphics[width=\textwidth]{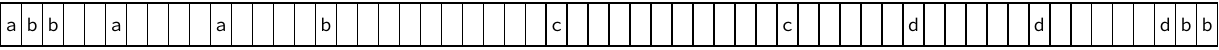}
    \caption{Word constructed in Theorem $\ref{wichmannword4}$ with parameters $r=s=2$. Empty space denotes a hole.}
    \label{wichmannword22}
\end{figure}

\begin{theorem} $\mathrm{HB}_4(n) \geq n-\sqrt{3n-3}-4$.
\end{theorem}
\begin{proof}
    For $n>213$, this follows from the Theorems $\ref{wichmannword4}$ and $\ref{useable upper bound}$ and the fact that $\mathrm{D}(W(r,s,i,j))=w_1(r,s,i,j)$ and $\mathrm{D}(W(r,s))=w_1(r,s)$.
    
    For $n \leq 213$ the constant $-4$ is so generous that the bound can be confirmed by easier constructions, for example with 
    $$a^{\lfloor \sqrt{n} \rfloor-1} b (\hole^{\lfloor \sqrt{n} \rfloor-1} c)^{t_1} \hole^{t_2}c$$ where $t_1$ and $t_2$ are chosen so that the length of the word is $n$ and $t_2 \leq \lfloor \sqrt{n} \rfloor-1$.
\end{proof}

For many values of $n$, Wichmann rulers and thus Wichmann words achieve the optimal values for $\mathrm{M}_1(n-1)$ and $\mathrm{HB}_4(n)$. It is conjectured that in rulers with $m$ marks, $m>13$, Wichmann rulers achieve the longest length. 

\section{Taking these ideas into 2 dimensions}

Ideas presented generalize easily to higher dimensions. We briefly visit the world of 2D unbordered partial words and 2D sparse rulers. We ask similar optimization questions as in 1D, and we present some preliminary results.

\begin{definition} \label{new definition}
   A complete sparse ruler of length $n$ and height $m$ in $2$ dimensions (later just $(n,m)$-ruler) is a subset $R$ of $\{0,1,\ldots n-1 \} \times \{0,1,\ldots m-1 \}$ so that for all $(x,y) \in \Z^2$, if $\abs{x}\leq n-1 $ and $\abs{y} \leq m-1$, then there exist $r_1,r_2 \in R$ so that $r_1 - r_2 = (x,y)$.
\end{definition}

We can informally define two-dimensional partial words to be rectangles consisting of alphabet symbols. We then can define a notion of border of two dimensional partial word as a part of a two dimensional word that can overlap with itself. This is formally defined in \cite{Avoiding-overlaps-in-pictures}. The correspondence Theorem $\ref{correspondense theorem}$ would work also in 2-dimensional case with the same proof, so instead we can directly define 2-dimensional partial word to be unbordered if its domain is a two-dimensional ruler that can measure non-zero vectors so that no measurement has to be done between same letters.

\begin{definition}
    Two-dimensional partial word $w$ is unbordered if the domain $\mathrm{D}(w)$ is a complete two-dimensional ruler and for all non-zero vectors $v$ that the ruler $\mathrm{D}(w)$ can measure, there exists $r_1$ and $r_2$ in $\mathrm{D}(w)$ so that $w(r_1) \neq w(r_2)$ and $r_1 - r_2 =v$.
\end{definition}

Like in the one dimensional case, we define $\mathrm{M}_2(n,m)$ to be the minimal number of marks in $(n,m)$-ruler. Now the same game of finding good rulers and then turning them to unbordered partial words over some finite alphabet can be played in 2D.

It is easy to see that cartesian product of two 1D-rulers is a 2D-ruler. This gives us one way to construct 2D-rulers. However, as one might expect, this is not the optimal way. The asymptotic we get from cartesian product of two 1D rulers is $$\mathrm{M}_2 (n,m) \leq 3\sqrt{mn} + \mathrm{O}(\sqrt{n}+\sqrt{m})$$ (assuming the 1D rulers we currently know are close to optimal). This growth rate is achieved by many simpler constructions, for example U-shaped walls form a better ruler than this, see \cite{soveltajat2dviivoittimet} for that construction.
The following construction gives the best growth rate among the constructions we found. It is the 2D version of the simple square root construction from 1D as seen in figure $\ref{2Drulerpicture}$. We also include a basic lower bound for $\mathrm{M}_2(n,m)$.
\begin{theorem}
    \label{best construction}
    $$\sqrt{4mn - 2(m+n)+\frac{1}{4}}+\frac{1}{2}\leq \mathrm{M}_2(n,m)\leq 2\sqrt{2}\sqrt{mn} + \mathrm{O}(\sqrt{n}+\sqrt{m})$$
    \end{theorem}
    \begin{proof}
        For the upper bound, following construction gives a working 2D-ruler. Let $R =R_1 \cup R_2 \cup R_3 \cup R_4$ where:
        \begin{align*}
            R_1 &= \{0,2,\ldots l\}\times \{1,2,\ldots k-1\} \\
            R_2 &= R_1+(n-l,0) \\
            R_3 &= \{(x,y) \vert x= 0 \mod l \text{ and } y = 0 \mod k\} \\
            R_4&=\{(n,m),(n,y), (x,m) \vert x=0 \mod l \text{ and } y=0 \mod k\}
        \end{align*} for some $1\leq l\leq n$ and $1 \leq k \leq m$. Figure $\ref{2Drulerpicture}$ demonstrates this ruler, and it is easy to see from the figure that in fact can measure all vectors. Counting the number of elements in each set we have $\abs{R_1} =\abs{R_2}= lk$, $\abs{R_3} = \lfloor \frac{nm}{lk} \rfloor$ and $\abs{R_4} \leq \lceil \frac{n}{l} +  \frac{m}{k} \rceil$.
        Now the number of elements in this ruler is 
        $$\abs{R} \leq \abs{R_1} + \abs{R_2}+\abs{R_3}+ \abs{R_4} \leq 2lk + \frac{nm}{lk} + \frac{n}{l}+ \frac{m}{k} +1.$$ Here if we pick $l=\lfloor \frac{\sqrt{n}}{\sqrt[4]{2}}\rfloor$ and $k=\lfloor \frac{\sqrt{m}}{\sqrt[4]{2}}\rfloor$ we reach the claimed number of marks.

        For the lower bound, first note that when a 2D-ruler can measure vector $(x,y)$, it then can also measure vector $(-x,-y)$ with the same marks. Let $v_{n,m}$ be the number of vectors that $(n,m)$-ruler needs to measure with pairwise different mark pairs. This number can be counted by first considering vectors without a zero component. We have the vectors $(x,y)$ where $x \in \{1,2,\ldots n-1\}$ and  $y \in \{1,2,\ldots m-1\}$. Then by changing those vectors into $(-x,y)$ we get another set of vectors that need to be measured differently from the previous vectors. Finally, we have $n+m$ vectors where one coordinate is zero. Together, this makes up $2(n-1)(m-1) + n + m = 2nm -n -m$ vectors that have to be measured by different mark pairs. The number of mark pairs then has to be greater than this, so $$ \binom{\mathrm{M}_2(n,m)}{2} \geq 2nm -n -m,$$ leading to the lower bound.
    \end{proof}

To make this construction into a two dimensional unbordered word, we can fill marks in $R_1$ and $R_2$ with a single letter, say $a$, and $R_3$ and $R_4$ with another letter, say $b \neq a$. The problem is that the sets have common elements, and we cannot assign multiple letters to single coordinate. To solve this problem, we can fill $R_1$ and $R_2$ with the first letter and add additional markers to measure the distances we miss by doing so. This means that we have to give the construction new ways to measure vectors that are of the form $r_2 - (0,0)$ and $(n,0)-r_1$, where $r_2 \in R_2$ and $r_1 \in R_1$. This can be done by adding letters on top of the areas $R_1$ and $R_2$. The number of added marks is $\mathrm{O}(\sqrt{n}+\sqrt{m})$, so we keep the same growth rate of the construction with even binary alphabet. This is shown more clearly in figure $\ref{binaryword2Druler}$. We summarise these observations in theorem below, where $\mathrm{HB}2_k(n,m)$ denotes the maximum number of holes a two dimensional $n$ by $m$ unbordered partial word over $k$-sized alphabet can have.

\begin{theorem} For all $k \geq 2$ and $mn\geq 1$,
    $$ nm-2\sqrt{2}\sqrt{nm}-\mathrm{O}(\sqrt{n}+\sqrt{m}) \leq \mathrm{HB}2_k(n,m)$$
    \label{2D words}
\end{theorem}
By adapting the arguments from one dimensional case in \cite{Combinatorics-on-partial-word-borders}, we can asymptotically solve the correct growth rate for $\mathrm{HB}2_2$. The argument turns the unbordered partial word to a graph, and then applies Turán's theorem.

\begin{theorem}[Turán's theorem \cite{turan1941external}]
    Given graph $G=(V,E)$ without $k+1$-clique, then $$\abs{E} \leq \frac{k-1}{2k} \abs{V}^2$$ holds.
    \label{turanstheorem}
\end{theorem}

\begin{theorem}[Essentially in  \cite{Combinatorics-on-partial-word-borders}] For all $k \geq 2$ and $mn\geq 1$
$$\mathrm{HB2}_k(n,m) \leq mn- \sqrt{\frac{2k}{k-1}(2nm-n-m)}.$$
\label{2D upperbound}
\end{theorem}
\begin{proof}
    In the proof of Theorem $\ref{best construction}$ we saw that $(n,m)$-ruler needs to have at least $2nm-n-m$ different pairs of marks. In unbordered partial words, this means that the word needs at least that many different pairs of letters such that pairs do not have the same letter. Consider the unbordered partial word over $k$-sized alphabet as a graph, where the nodes are the letters and there is edge between two letters if the letters are not compatible. Then this graph is obviously $k$-colorable and thus has no $k+1$-clique. By Theorem $\ref{turanstheorem}$, we then have $$2nm- n-m \leq \frac{k-1}{2k}(mn-\mathrm{HB2}_k(n,m))^2,$$ where by solving for $\mathrm{HB2}_k(n,m)$ we get the proof of the theorem.
\end{proof}
Combining Theorems $\ref{2D words}$ and $\ref{2D upperbound}$ we see that $$\mathrm{HB}2_2(n,m) = mn- 2\sqrt{2}\sqrt{mn} + \mathrm{O}(\sqrt{n}+\sqrt{m}).$$

\begin{figure}[!h]
    \centering
    \includegraphics[width=\textwidth]{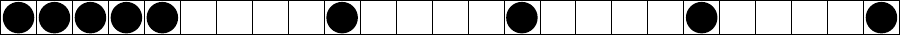}
    \caption{One dimensional ruler of length $24$ (so there is $25$ ``slots'').}
    \includegraphics[width=\textwidth]{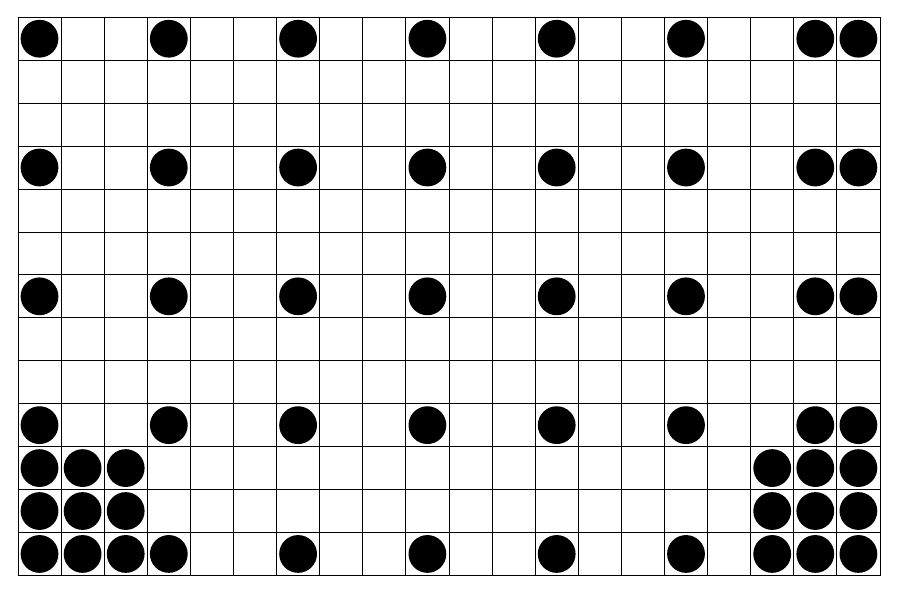}
    \caption{$(20,13)$-ruler constructed in Theorem $\ref{best construction}$ with parameters $l=k=3$. Note that this is just 2D version of simple 1D construction above.}
    \label{2Drulerpicture}
\end{figure}

\newpage

\begin{figure}[!h]
    \centering
    \includegraphics[width=\textwidth]{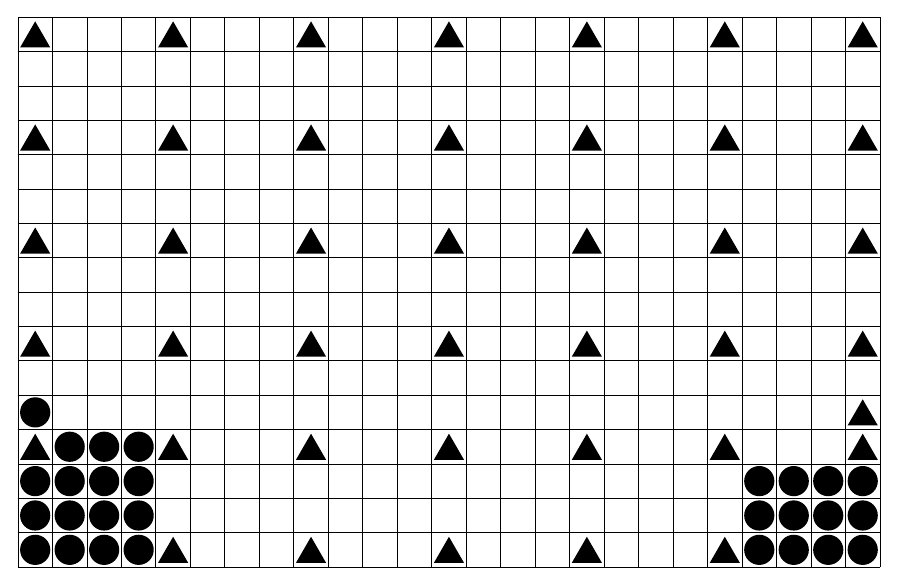}
    \caption{A 2D ruler turned into 2D unbordered binary partial word. Here the alphabet consists of disks and triangles and an empty cell denotes hole. Note the few extra marks placed on top of the rectangle areas compared to figure $\ref{2Drulerpicture}$.}
    \label{binaryword2Druler}
\end{figure}

\section{Quick note on cross-bifix-free codes}

Lastly we note a quick observation related to so called cross-bifix-free codes of fixed length. A code $C$ is cross-bifix-free if for any codewords $c_1 \in C$ and $c_2 \in C$ there is no proper prefix of $c_1$ which is also suffix of $c_2$. More formally $\forall c_1,c_2 \in C$:
$$ c_1=xy_1 \land c_2=y_2x \implies c_1=c_2=x \lor x=\varepsilon.$$
These codes and specially the number of codewords in maximal cross-bifix-free codes have seen recent interest, see \cite{stanovnik2024search} for a summary of recent results.

Related to our topic, given alphabet $A$ and a partial word $w \in A_\hole^*$ we can define a language by filling all the holes in the given word by letters.
\begin{definition}
    $$L_\uparrow (w) = \{u  \in A^* \vert w \uparrow u \}$$
\end{definition}
 
If $w \in A_{\hole}^*$ is an unbordered partial word with $k$ holes,
then it is easy to see that $L_{\uparrow}(w)$ is a cross-bifix-free code of size $|A|^k$.
In the binary case, for example, this idea would give a subset of $A^n$ of size $2^{\lfloor n-2\sqrt{n-1} \rfloor}$ by taking $w$ to be one of the optimal unbordered binary partial words found in \cite{Combinatorics-on-partial-word-borders}.
It is known that the maximal size of a cross-bifix-free subset of $A^n$ is much larger,
see~\cite{stanovnik2024search},
but the partial word construction is worth mentioning because it is particularly simple,
and also because it gives another connection between different research topics.

In the two-dimensional case, this idea becomes perhaps more interesting
because the previously known results are not as strong as in the one-dimensional case.
Again, if $w$ is an unbordered two-dimensional partial word with $k$ holes,
then $L_{\uparrow}(w)$ is a cross-bifix-free code of size $|A|^k$.
Therefore, we can use Theorem \ref{2D words}
to get a cross-bifix-free set of binary $n$-by-$m$-words of size \\ $2^{nm-2\sqrt{2}\sqrt{nm}-\mathrm{O}(\sqrt{n}+\sqrt{m})}$, where the exact value of the exponent is easy to calculate once the seed word is fixed.
This could be compared with existing constructions, such as the ones in \cite{anselmo2017non},
\cite{barcucci2017non},
\cite{barcucci2021non},
and we suspect it could be used as a basis for a more complicated construction.
A detailed comparison and a study of improved constructions
is a potential direction for future research.

\section{Conclusions}
In sections $2$ and $3$ we showed a new connection between unbordered partial words and complete sparse rulers. Using this connection and earlier ruler results, we get new bounds on the function $\mathrm{HB}_k$ when $k\geq4$. The new bounds are
$$ n-\sqrt{3n-3}-4 \leq \mathrm{HB}_4 (n) \leq \mathrm{HB}_k (n) \leq \mathrm{HB}_\infty (n) \leq n- \sqrt{2.43n}, $$ where the lower bound asymptotically beats previous lower bounds and the upper bound does so when $k\geq 6$.

We also provided a direct counterexample to the claim $\mathrm{HB}_k(n+1) \leq \mathrm{HB}_k(n) \\ +1$ for all $n$ and $k$. The counterexample shows that this is not the case for $k\geq4$. Interestingly for $k=2$ the claim holds, but the case $k=3$ is unclear.
\begin{conjecture} For all $n\geq 1$
$$\mathrm{HB}_3(n+1) \leq \mathrm{HB}_3(n)+1.$$
    
\end{conjecture}

Given that the best known ruler construction can be converted to a word over an alphabet of size $4$ and the ruler construction is thought to be optimal, it really seems that the four letter alphabet gives the optimal unbordered partial words. This was conjectured already by Blanchet-Sadri et al. \cite{blanchet2009many}, and now the conjecture has new interest in the study of sparse rulers.
\begin{conjecture}$\label{alphabetsize conjecture}$ For all $n\geq1$
    $$\mathrm{HB}_4(n) = \mathrm{HB}_\infty(n)$$
\end{conjecture}

We then in section $4$ considered two dimensional generalizations of these ideas. Unsurprisingly in two dimensions it is easy to come up with asymptotically better rulers than just a cartesian product of two one-dimensional rulers. We showed one example, which could also be turned into binary unbordered two dimensional words with just a few extra marks. This reaches the optimal asymptotic for $n$ by $m$ unbordered binary partial word when both $n$ and $m$ are taken to infinity. Without explicit conjecture, we state the open problem of finding better two-dimensional rulers and unbordered partial words.
\bibliographystyle{splncs04}
\bibliography{ref}

\begin{thebibliography}{10}
\providecommand{\url}[1]{\texttt{#1}}
\providecommand{\urlprefix}{URL }
\providecommand{\doi}[1]{https://doi.org/#1}

\bibitem{sparsearraysovellus}
Aboumahmoud, I., Muqaibel, A., Alhassoun, M., Alawsh, S.: A review of sparse sensor arrays for two-dimensional direction-of-arrival estimation. IEEE Access  \textbf{9},  92999--93017 (2021). \doi{10.1109/ACCESS.2021.3092529}

\bibitem{Combinatorics-on-partial-word-borders}
Allen, E., Blanchet-Sadri, F., Bodnar, M., Bowers, B., Hidakatsu, J., Lensmire, J.: Combinatorics on partial word borders. Theoret. Comput. Sci.  \textbf{609},  469--493 (2016). \doi{10.1016/j.tcs.2015.11.006}

\bibitem{weakperiod}
Allen, E., Blanchet-Sadri, F., Byrum, C., Cucuringu, M., Merca{\c{s}}, R.: Counting bordered partial words by critical positions. The electronic journal of combinatorics  \textbf{18}(1), ~138 (2011). \doi{10.37236/625}

\bibitem{Avoiding-overlaps-in-pictures}
Anselmo, M., Giammarresi, D., Madonia, M.: Avoiding overlaps in pictures. In: 19th International Conference on Descriptional Complexity of Formal Systems (DCFS). pp. 16--32. Springer International Publishing (2017). \doi{10.1007/978-3-319-60252-3_2}

\bibitem{anselmo2017non}
Anselmo, M., Giammarresi, D., Madonia, M.: Non-expandable non-overlapping sets of pictures. Theoretical Computer Science  \textbf{657},  127--136 (2017)

\bibitem{barcucci2017non}
Barcucci, E., Bernini, A., Bilotta, S., Pinzani, R.: Non-overlapping matrices. Theoretical Computer Science  \textbf{658},  36--45 (2017)

\bibitem{barcucci2021non}
Barcucci, E., Bernini, A., Bilotta, S., Pinzani, R., et~al.: Non-overlapping matrices via dyck words. ENUMERATIVE COMBINATORICS AND APPLICATIONS.  \textbf{1} (2021)

\bibitem{bernshteyn2019improved}
Bernshteyn, A., Tait, M.: Improved lower bound for difference bases. Journal of Number Theory  \textbf{205},  50--58 (2019). \doi{10.1016/j.jnt.2019.05.002}

\bibitem{berstel1999partial}
Berstel, J., Boasson, L.: Partial words and a theorem of fine and wilf. Theoretical Computer Science  \textbf{218}(1),  135--141 (1999). \doi{10.1016/S0304-3975(98)00255-2}

\bibitem{blanchet2009many}
Blanchet-Sadri, F., Allen, E., Byrum, C., Merca{\c{s}}, R.: How many holes can an unbordered partial word contain? In: International Conference on Language and Automata Theory and Applications. pp. 176--187. Springer (2009)

\bibitem{hittingallthemarks}
E. Pegg, "Hitting All the Marks: Exploring New Bounds for Sparse Rulers and a Wolfram Language Proof." https://blog.wolfram.com/2020/02/12/hitting-all-the-marks-exploring-new-bounds-for-sparse-rulers-and-a-wolfram-language-proof/

\bibitem{erdos1948representation}
Erd{\"o}s, P., G{\'a}l, I.: On the representation of 1, 2,..., n by differences. Indagationes Math  \textbf{10}(9),  379--382 (1948)

\bibitem{halava2008square}
Halava, V., Harju, T., K{\"a}rki, T.: Square-free partial words. Information Processing Letters  \textbf{108}(5),  290--292 (2008). \doi{10.1016/j.ipl.2008.06.001}

\bibitem{halava2009overlap}
Halava, V., Harju, T., K{\"a}rki, T., S{\'e}{\'e}bold, P.: Overlap-freeness in infinite partial words. Theoretical Computer Science  \textbf{410}(8-10),  943--948 (2009). \doi{0.1016/j.tcs.2008.12.041}

\bibitem{Leech}
Leech, J.: On the representation of 1,2,\ldots n by differences. Journal of the London Mathematical Society  \textbf{s1-31}(2),  160--169 (1956). \doi{10.1112/jlms/s1-31.2.160}

\bibitem{soveltajat2dviivoittimet}
Liu, C.L., Vaidyanathan, P.P.: Two-dimensional sparse arrays with hole-free coarray and reduced mutual coupling. In: 2016 50th Asilomar Conference on Signals, Systems and Computers. pp. 1508--1512 (2016). \doi{10.1109/ACSSC.2016.7869629}

\bibitem{ruler-sequence}
OEIS Foundation Inc., Entry A046693 in The On-Line Encyclopedia of Integer Sequences, (2024) https://oeis.org/A046693

\bibitem{redei1949representation}
R{\'e}dei, L., R{\'e}nyi, A.: On the representation of the numbers 1, 2,..., n by means of differences. Mat. Sbornik NS  \textbf{24}(66),  385--389 (1949)

\bibitem{stanovnik2024search}
Stanovnik, L., Mo{\v{s}}kon, M., Mraz, M.: In search of maximum non-overlapping codes. Designs, Codes and Cryptography pp. 1--28 (2024). \doi{10.1007/s10623-023-01344-z}

\bibitem{turan1941external}
Tur{\'a}n, P.: On an external problem in graph theory. Mat. Fiz. Lapok  \textbf{48},  436--452 (1941)

\bibitem{Gradu}
Vanhatalo, A.: Reunattomien osittaissanojen ja harvojen viivoittimien yhteydest{\"a} (2024), \url{https://urn.fi/URN:NBN:fi-fe2024032713407}, master's thesis, University of Turku

\bibitem{Wichmann}
Wichmann, B.: A note on restricted difference bases. Journal of the London Mathematical Society  \textbf{s1-38}(1),  465--466 (1963). \doi{10.1112/jlms/s1-38.1.465}

\end{thebibliography}

\end{document}